\documentclass[a4paper,12pt, twoside, reqno]{amsart}
\usepackage{amsthm}
\usepackage{amsmath}
\usepackage{amssymb}
\usepackage{eucal}

\newtheorem{thm}{Theorem}[section]%%[chapter]
\newtheorem{coro}[thm]{Corollary}
\newtheorem{prop}[thm]{Proposition}

\newtheorem{lemma}[thm]{Lemma}

\theoremstyle{definition}

\newtheorem{exm}[thm]{Example}

\catcode`@=11
\catcode`@=12

\let\emptyset=\varnothing

\begin{document}

\title{Two open problems in the fixed point theory of contractive type mappings on first-countable quasimetric spaces}
\author{Mitrofan M. Choban$^{1}$ and Vasile Berinde$^{2,3}$}

%MSC: 47H10; 54H25
\medskip
%Keywords: distance space, coincidence point, fixed point
\begin{abstract}

Two open problems in the fixed point theory of quasi metric spaces posed in [Berinde, V. and   Choban, M. M.,  {\it Generalized distances and their associate metrics.
Impact on fixed point theory},  Creat. Math. Inform. {\bf 22}
(2013), no. 1, 23--32] are considered.
We give a complete answer to the first problem, a partial answer to the second one, and also illustrate the complexity and relevance of these problems  by means of four very interesting and comprehensive examples.
\end{abstract}

\maketitle

\pagestyle{myheadings} \markboth{Mitrofan M. Choban and Vasile Berinde} {Two open problems}

 \section{Introduction and Preliminaries}

The exist many generalizations of contraction principle in literature, which are established in various settings: cone metric spaces, quasimetric spaces (or $b$-metric spaces), partial metric spaces, $G$-metric spaces, $w$-metric spaces, $\tau$-metric spaces etc. It is really difficult to delineate the true generalizations of the trivial ones. In some recent papers \cite{Hag11}, \cite{Hag13}, \cite{SamVV}, the authors tried to differentiate, amongst  this rich literature, which results are true generalizations and which are trivial. They pointed out some such trivial generalizations in the case of cone metric spaces and partial metric spaces, see \cite{Hag11}, \cite{Hag13}), while in \cite{SamVV}, the authors studied the same problem but for $G$-metric spaces.

This problem arose as a natural reaction to the flood of fixed point research papers published in the last decade. In a recent paper \cite{BC2}, the present authors  inspected whether or not a similar situation to that reported in  \cite{Hag11}, \cite{Hag13} and \cite{SamVV} may happen in the case of $b$-metric spaces (also called quasimetric spaces) and concluded that working in $b$-metric spaces makes sense since, if $ \rho: X\times X \rightarrow \mathbb{R}$ is a quasimetric, then the associate 
functional $\bar \rho: X\times X \rightarrow \mathbb{R}$ generated by
 $\rho$ and given by 
 $$
 \bar \rho (x,y) = \inf \{\rho (x,z_1)+...+\rho (z_i,z_{i+1})+\dots
 $$
\begin{equation} \label{eq-1}
 +\rho (z_n,y):
n \in \mathbb N, z_1,\dots,z_n \in X\},
\end{equation}
 is in general not a metric. The paper  \cite{BC2} naturally closes with the following two open problems.
 
 {\bf Problem 1.} {\it Let $g: X \longrightarrow X$ be a contraction on a complete quasimetric space $(X, d)$.
 Is it true that $g$ has fixed points?}

 {\bf Problem 2.} {\it Let $g: X \longrightarrow X$ be a contraction of a complete $F$-symmetric space $(X, d)$.
 Is it true that $g$ has fixed points?}
 
 As, to our best knowledge, Problems 1 and 2 remained open so far,  it is our aim in this paper to give positive answers to them and also to provide some examples 
 illuminating  to some extent the complexity of the problems.

Throughout the paper, by a space we understand a  topological  $T_0$-space, and we use the terminology from \cite{Eng, GD, RP}.

Let $X$ be a non-empty set and $d : X\times X \rightarrow \mathbb R$ be a mapping such that:

($i_m$) $d(x, y) \geq  0$, for all $x, y \in X$;

($ii_m$) $d(x, y) + d(y,x) = 0$ if and only if $x = y$.

Then $(X, d)$ is called a {\it distance space} and $d$ is called a {\it distance} on $X$.

Let $d$ be a distance on $X$ and $B(x,d,r)$ = $\{y \in X: d(x, y) < r\}$ be the {\it ball} with the center $x$
and radius $r > 0$. The set $U \subset X$ is called {\it $d$-open} if for any $x \in U$ there exists $r > 0$
such that $B(x,d,r) \subset U$. The family $\mathcal T(d)$ of all $d$-open subsets is the topology on $X$ 
generated by $d$. The space $(X, \mathcal T(d))$ is a $T_0$-space.

 A distance space is a {\it sequential space}, i.e., a set $B \subseteq X$ is closed if and only 
if, together with any sequence, $B$ contains all its limits (see \cite{Eng}).

  Let $(X, d)$ be a  distance space, $\{x_n: n \in \mathbb N =\{1, 2,...\}\}$ 
be a sequence in $X$
 and a point $x \in X$. We say that the sequence   $\{x_n: n \in \mathbb N\}$ is:

1) {\it convergent} to $x$ if and only if $\lim_{n\rightarrow \infty }d(x, x_n) = 0$. 
We denote this by $x_n\rightarrow x$ or $x = \lim_{n\rightarrow \infty }x_n$.

2)  {\it Cauchy} or fundamental if $\lim_{n, m\rightarrow \infty }d(x_n, x_m) = 0$.

    We say that  a  distance space $(X, d)$ is {\it complete} if  every
Cauchy sequence in $X$ converges to some point   in $X$.
If  $d $ is a distance on $X$ such that:
\vspace{0.1cm}

($iii_m$) $d(x, y)$ = $d(y, x)$, for all $x, y \in X$,
\vspace{0.1cm}

\noindent then $(X, d)$ is called a {\it symmetric space} and $d$ is called a {\it symmetric} on $X$.
If  $d $ is a distance on $X$ such that:
\vspace{0.1cm}

($iv_m$) $d(x, z) \leq d(x, y) + d(y, z)$, for all $x, y, z \in X$, 
\vspace{0.1cm}

\noindent then $(X, d)$ is called a {\it quasimetric space} and $d$ is called a {\it quasimetric} on $X$.

  A distance $d$ on a set $X$ is called a {\it metric} if  it is  simultaneously a symmetric and a quasimetric.
  
  Let $X$ be a non-empty set and  $d(x, y) $ be a distance on $X$
   with the following property:
\vspace{0.1cm}

(N) for each point $x \in X$ and any $\varepsilon > 0$ there exists $\delta = \delta (x,\varepsilon ) > 0$
such that from $d(x, y) \leq \delta $ and $d(y,z) \leq \delta $ it follows $d(x, z) \leq \varepsilon $.
\vspace{0.1cm}

Then $(X, d)$ is called an {\it N-distance space} and $d$ is called an {\it N-distance}
on $X$.
If $d$ is a symmetric, then we say that $d$ is an $N$-symmetric.

If $d$ satisfies the condition 
\vspace{0.1cm}

(F) for any $\varepsilon > 0$ there exists $\delta = \delta (\varepsilon ) > 0$
such that from $d(x, y) \leq \delta $ and $d(y,z) \leq \delta $ it follows $d(x, z) \leq \varepsilon $,
\vspace{0.1cm}

\noindent then  $d$ is called an {\it F-distance}  or a {\it Fr\'echet distance}  and  $(X, d)$ is 
called an {\it F-distance space}.
 Obviously, any $F$-distance $d$ is an $N$-distance, too, but the reverse is not true, in general, see Examples 1.1 and 1.2 in \cite{C1}.
 
 A distance space $(X, d)$ is called an {\it H-distance space} if for any two distinct points $x, y \in X$
 there exists $\delta  = \delta (x, y) > 0$ such that $d(x,z)$+ $d(y,z) \geq \delta $ for each point $z \in X$,
i.e.  $B(x, d, \delta ) \cap B(y, d, \delta ) = \emptyset $.

 Any $N$-symmetric $d$ is an $H$-distance, too.
A space $(X, d)$ is a $H$-distance space if and
only if any convergent  sequence has a unique limit point (see \cite{SN}, Theorem 3). 

 \section{ Conditions ensuring the existence of fixed points}  \label{sect}

Consider the mapping $\varphi : X \longrightarrow  X$ and let $\varphi ^1$ = $\varphi $ and $\varphi ^{n+1}$ = $\varphi \circ \varphi ^n$ for each $n \in \mathbb N$ = $\{1, 2, ...\}$. Denote by $Fix\,(\varphi )$ the set of fixed points of $\varphi$.
If $x \in X$, then we put $x_0$ = $x$ and $x_n $ = $\varphi ^n(x)$,   for every $n \in \mathbb N$. The set $O(x,\varphi )$ = $\{x_n: n \in \mathbb N\}$  is commonly called the Picard orbit of $\varphi$ at the point $x$.

A mapping $\varphi : X \rightarrow X$ is called:

(i) {\it Lipschitzian} or $\lambda$-{\it Lipschitzian} if there exists  $\lambda  > 0$ such 
that 
\begin{equation} \label{contr}
d(\varphi (x), \varphi (y)) \leq  \lambda \cdot d(x,y), \textnormal{ for all  } x, y \in  X;
\end{equation}

(ii) {\it contraction} or $\lambda$-{\it contraction} if  it is $\lambda$-{\it Lipschitzian} with $0\leq \lambda  < 1$;

(iii) {\it nonexpansive}  if  it is $\lambda$-{\it Lipschitzian} with $\lambda  = 1$.

\begin{prop}\label{P2.1}    Let $(X, d)$ be a H-distance space, $\varphi : X \longrightarrow X$ be 
a $\lambda$-Lipschitzian or a continuous mapping. Suppose that, for
some point $x_0 \in X$,  the Picard sequence $O(x_0, \varphi )$ is convergent. 

Then the mapping $\varphi $ is continuous
and  $Fix\,(\varphi )\neq \emptyset$. 
\end{prop}  

\begin{proof} Assume  that the mapping $\varphi $ is $\lambda$-Lipschitzian. Since $\varphi (B(x, d, (1+\lambda )^{-1}r) \subseteq B(\varphi (x), d, r)$ 
for any point $x \in X$ and any number $r > 0$,  the mapping $\varphi $ is continuous.

 Let  $\{x_n = \varphi ^n(x) \in X: n \in \mathbb N\}$ be the Picard sequence of $\varphi$ at the given
point $x_0 \in X$, which, by hypothesis, converges to a point $a \in X$.
Then, since the mapping $\varphi $ is continuous and $\lim_{n\rightarrow \infty }d(a,x_n)$ = $0$, we have 
 $\lim_{n\rightarrow \infty }d(\varphi (a),\varphi (x_n))$ = $\lim_{n\rightarrow \infty }d(\varphi (a),x_n)$ = $0$ and
 $\lim_{n\rightarrow \infty }x_n$ = $\varphi (a)$. Hence   $\varphi (a)$ = $a$. 
 \end{proof} 
 
 \begin{thm}\label{T2.2}     Let $d$ be  simultaneously an $N$-distance and an $H$-distance on a space $X$ and let $\varphi : X \longrightarrow X$
 be a mapping with the following properties:

\indent (i) $\varphi $ is continuous or $\lambda$-Lipschitzian;

(ii) for some point $e \in X$,  $O(e, \varphi )$ = $\{e_n =\varphi ^n(e): n \in \mathbb N\}$ 
has an accumulation point and $\lim _{n\rightarrow \infty }d(e_n,e_{n+1})  = 0$.
     
     Then:

1. $Fix\,(\varphi)\neq \emptyset$  and any accumulation point of the orbit   $O(e, \varphi )$ is a fixed pout of $\varphi$.

2.  The orbit  $O(e, \varphi )$  has not periodic points.

3. If  $\lim _{n\rightarrow \infty }d(g^n(y),g^{n+1}(y))  = 0$, for each point $y \in X$, then
any periodic point of the mapping $\varphi $ is a fixed point of $\varphi $.
 
4. The space $(X, \mathcal T(d))$ is first-countable and Hausdorff.
\end{thm}

\begin{proof} 
From Proposition \ref{P2.1},
it follows that  $\varphi $ is continuous. 
Fix $r > 0$ and $a \in X$. There exists $\delta  > 0$ such that from $d(a, x) \leq \delta $ and $d(x, y) \leq \delta $
it follows that $d(a, y) < r$. Hence $d(x, y) > r$ provided $d(a,x) \leq \delta $ and $y \not\in B(x,d, r)$. From Theorem 4 in \cite{SN}
it follows that $(X, \mathcal T(d))$ is a first-countable space. Hence $a \in cl_XB$ if and only if $d(a, B)$ =
$\inf \{d(a, x):  x \in B\}$ = $0$. A first-countable space with  an $H$-distance is Hausdorff and hence  $d(x, y)$ = $0$ if and only if $x $  = $y$.

Fix $x \in X$. Let   $O(x, \varphi )$ = $\{x_n =\varphi ^n(x): n \in \mathbb N\}$ be the Picard orbit of $\varphi $
at the point $x$. Suppose that   $\lim _{n\rightarrow \infty }d(x_n,x_{n+1})  = 0$. 
Assume that $ x_k  = x_{k+m}$  for some $k, m \in \mathbb N$ and  $m \geq 1$. We have $x_k$ = $x_{k+nm} \not= x_{k+nm+1}$ =  $x_{k+1}$, 
which contradicts the condition  $\lim _{n\rightarrow \infty }d(x_n,x_{n+1})  = 0$. Hence the mapping $\varphi $
has no periodic non-fixed points in the condition that  $\lim _{n\rightarrow \infty }d(g^n(y),g^{n+1}(y))  = 0$ for each point $y \in X$.
In particular, the Picard orbit of $\varphi $ at the point $e$ has no periodic non-fixed points.

If $b \in X$ and $b = e_n = e_{n+1}$  for some $n \in \mathbb N$, then $b$ is a  fixed point of the mapping $\varphi $  and  $O(x,\varphi )$ 
is a Cauchy sequence with the accumulation point  $b$. In this case the assertions of theorem are proved.

 Assume now that $ e_n \not= e_{n+m}$,  for any $n, m \in \mathbb N$.  In this case the set
     $O(e, \varphi )$ is infinite and non-closed in the sequential space  $(X, \mathcal T(d))$.
     Then there exist a point $b \in X$ and a sequence  $\{n_k \in \mathbb N: k \in \mathbb N\}$ such that
$b = \lim_{k \rightarrow \infty } e_{n_k}$, $n_k < n_{k+1}$ and $d(b, e_{n_{k+1}}) < d(b, e_{n_k}) < 2^{-k}$ 
for each $k \in \mathbb N$.     
     
     For each $\varepsilon > 0$ there exists $\delta = \delta (b,\varepsilon ) > 0$
such that from $d(b, y) \leq \delta $ and $d(y,z) \leq \delta $ it follows $d(b, z) \leq \varepsilon $.
We assume that $2\delta  < \epsilon $.
We put $c = \varphi (b)$, $y_k = e_{n_k}$ and $z_k = \varphi (y_k)$. Then $b = \lim_{k\rightarrow \infty} y_k$ 
and, since the  mapping $\varphi  $ is continuous, $c = \lim_{k\rightarrow \infty} z_k$.

We claim that  $b= \lim_{k\rightarrow \infty} z_k$. Fix $\varepsilon  > 0$. There exists $\delta  > 0$
such that:

a)  $d(b, y) < \delta $ and $d(y, z) < \delta $ implies $d(b,z) < \varepsilon $;

b) $d(c, y) < \delta $ and $d(y, z) < \delta $ implies $d(c,z) < \varepsilon $.
\vspace{0.1cm}

Fix $n_1 \in \mathbb N$ such that $2^{-n_1} < \delta $.
Since $\lim _{n\rightarrow \infty }d(x_n,x_{n+1})  = 0$, there exists $m \in \mathbb N$ such that $m \geq  n_1$
and  $d(e_n,e_{n+1}) < \delta $ for each $n \geq m$. Then from $k \geq m$ we have $d(b,y_k) < \delta $, $d(y_k,z_k) < \delta $
and hence  $d(b,z_k) < \varepsilon $. Therefore,    $b= \lim_{k\rightarrow \infty} z_k$.

So, $b = c$ and $\varphi (b)$ = $b$. 
 \end{proof}

 \begin{thm}\label{T2.3}     Let $d$ be  simultaneously an $N$-distance and an $H$-distance on a space $X$ and $\varphi : X \longrightarrow X$ 
 be a contraction with the property that
there exists a point $a \in X$ such that  $O(a, \varphi )$ = $\{a_n =\varphi ^n(x): n \in \mathbb N\}$ 
has an accumulation point.

     Then:

1. The mapping $\varphi $ is continuous and  has a unique  fixed point. 

2.  Any periodic point of the mapping $\varphi $ is a fixed point of $\varphi $.
 
3. Any Picard orbit is  convergent to the fixed point.
\end{thm}

\begin{proof}  

Fix $r > 0$ and $a \in X$. There exists $\delta  > 0$ such that from $d(a, x) \leq \delta $ and $d(x, y) \leq \delta $
it follows that $d(a, y) < r$. Hence $d(x, y) > r$ provided $d(a,x) \leq \delta $ and $y \not\in B(x,d, r)$. From Theorem 4 in \cite{SN}
it follows that $(X, \mathcal T(d))$ is a first-countable space. Hence $a \in cl_XB$ if and only if $d(a, B)$ =
$\inf \{d(a, x):  x \in B\}$ = $0$. But a first-countable space with  an $H$-distance is Hausdorff. 
This means that $d(x, y)$ = $0$ if and only if $x $  = $y$.

From Theorem \ref{T2.2}
it follows that: a) the mapping $\varphi $ is continuous; b) $\varphi $ has not two distinct fixed points;
c) any periodic point of $\varphi $ is a fixed point.

Fix $x, y \in X$. Let  $O(x, \varphi )$ = $\{x_n =\varphi ^n(x): n \in \mathbb N\}$ and  $O(y, \varphi )$ 
= $\{y_n =\varphi ^n(y): n \in \mathbb N\}$ be the Picard orbits
of $\varphi $ at the points $x$  and $y$.  Fix a number $\mu  > 0$ such that $d(x_1, x_2)$ + $d(x_2, x_1)$ +  $d(y_1, y_2)$ +  $d(y_2, 1_1)$
+  $d(x_1, y_1)$ +  $d(y_1, x_1)  < \mu $. Then $d(x_n, x_{n+1}) < \lambda ^n  \cdot  \mu $ and   $\lim _{n\rightarrow \infty }d(x_n,x_{n+1})  = 0$.
From the inequality $d(x_n, y_n)$ +  $d(y_n, x_n)  < \lambda ^n\cdot \mu $ it follows that the sequences  $O(x, \varphi )$ and   $O(y, \varphi )$ are  
the same accumulation points. Hence, any Picard orbit of $\varphi $ has accumulation points. On the other hand, by  Theorem \ref{T2.2}, any
accumulation point of a Picard orbit of $\varphi $ is a fixed point of $\varphi $. Thus the Picard orbits have a unique accumulation point $b$ = $\varphi (b)$. Let $\eta  > d(b,x_1)$ + $d(x_1, b)$.

Then  $d(b,x_n)$ + $d(x_n, b) < \lambda ^n \cdot \eta $ and hence $\lim_{n\rightarrow \infty } x_n $ = $b$.
 \end{proof}

\begin{coro}\label{C2.4}     Let $d$ be simultaneously a  quasimetric and an $H$-distance on a space $X$ and $\varphi : X \longrightarrow X$
 be a mapping with properties:

\noindent (i) $\varphi $ is continuous or 
$\lambda$-Lipschitzian;

\noindent (ii) for some point $e \in X$, the  Picard orbit  $O(e, \varphi )$ = $\{e_n =\varphi ^n(e): n \in \mathbb N\}$ 
has an accumulation point and $\lim _{n\rightarrow \infty }d(e_n,e_{n+1})  = 0$.%  and $\lim _{n\rightarrow \infty }d(e_n,e_{n+1})  = 0$.
     
     Then:

1. $Fix\,(\varphi )\neq \emptyset$ and any accumulation point of the orbit   $O(e, \varphi )$ is a fixed point of $\varphi $.

2.  The orbit $O(e, \varphi )$ has no periodic points.

3. If  $lim _{n\rightarrow \infty }d(\varphi^n(y),\varphi^{n+1}(y))  = 0$, for each point $y \in X$, then
any periodic point of the mapping $\varphi $ is a fixed point of $\varphi $.
 
4. The space $(X, \mathcal T(d))$ is first-countable and Hausdorff.
\end{coro}

\begin{coro}\label{C2.5}     Let $d$ be  simultaneously a complete quasimetric and an $H$-distance on a space $X$ and $\varphi : X \longrightarrow X$
 be a mapping with the properties:
 
\noindent (i) $\varphi $ is continuous or  
$\lambda$-Lipschitzian;
 
\noindent (ii)  for each point $x \in X$ and the  Picard orbit  $O(x, \varphi )$ = $\{x_n =\varphi ^n(x): n \in \mathbb N\}$ there exists
 a non-negative number  $\mu (x)  < 1$ such that $d(\varphi (x_n), \varphi (x_m)) \leq \mu (x)\cdot d(x_n,x_m)$ for all $n, m \in \mathbb N$.
     
     Then:

1. $Fix\,(\varphi )\neq \emptyset$.

2.  Any periodic point of the mapping $\varphi $ is a fixed point of $\varphi $.
 
3. Any Picard orbit is a Cauchy convergent sequence to some fixed point of $\varphi $.

4. The space $(X, \mathcal T(d))$ is first-countable and Hausdorff.
\end{coro}

 \begin{thm}\label{T2.4}     Let $d$ be  simultaneously a complete distance and an $H$-distance on a space $X$ and $\varphi : X \longrightarrow X$
 be a contraction with the property that   there exist two numbers $\delta  > 0$ and $a \geq 1$ such that from $d(x, y) \leq \delta $  and $d(y, z) \leq \delta $ it follows that
 $d(x, z) \leq a[d(x,y) $ + $d(y, z)]$.

     Then:

1. The mapping $\varphi $ is continuous and  has a unique  fixed point. 

2.  Any periodic point of the mapping $\varphi $ is a fixed point of $\varphi $.
 
3. Any Picard orbit is a Cauchy sequence convergent to the fixed point of $\varphi $.
\end{thm}

\begin{proof} As in the proof of Theorem 4.2 from \cite{C1}, we first prove that any Picard orbit is a Cauchy sequence.
Hence any Picard orbit is a Cauchy sequence convergent to some point. Now, Theorem \ref{T2.3} completes the proof.
 \end{proof}

 \section{ Examples}

 The first two examples in this section show that the requirement that $d$  is an $H$-distance on $X$ in Theorem \ref{T2.2}, Theorem \ref{T2.3}, Theorem \ref{T2.4} and in Corollaries \ref{C2.4} and  \ref{C2.4}  is essential.

 \begin{exm}\label{E3.1}  
Let $X$ = $\{a, b\} \cup \mathbb N$ be a countable set with distinct elements. Consider the distance $d:X\times X\rightarrow \mathbb{R}_{+} $, defined by:

\noindent (i)  $d(x,x)$ = $0$, for any $x \in X$;

\noindent (ii)  $d(m, n)$ = $d(n, m)$ = $|2^{-n} - 2^{-m}|$, for all $n, m \in \mathbb N \subseteq X$;

\noindent (iii)  $d(a, n)$ = $d(b, n)$ = $2^{-n}$, for each $n \in \mathbb N$;

\noindent (iv)   $d(n, a)$ = $d(n, b)$ = $d(a,b)$ = $1$, for each $n \in \mathbb N$.

Then $(X, d)$ is a quasimetric space but $d$ is not an $H$-distance, because for  $x=a, y=b$
 there is no $\delta  = \delta (x, y) > 0$ such that $d(x,n)$+ $d(y,n)=2^{-n+1} \geq \delta $, for all  $n \in \mathbb N$.
 
  Moreover, if we consider the mapping $\varphi : X \longrightarrow X$ defined by:
$\varphi (a)$ = $b \not= \varphi (b)$ = $a$ and $\varphi (n)$ = $n+1$, for each $n \in \mathbb N$, then any Picard orbit $O(n, \varphi )$ is a Cauchy convergent sequence, for each $n \in \mathbb N$, but $\varphi $ is fixed point free.
\end{exm}

\begin{exm}\label{E3.2}  Let $X$ = $\omega $ := $\{0, 1, 2, ...\}$.  On $X$ consider the
distance $d:X\times X\rightarrow \mathbb{R}_{+} $, defined by:
 
\noindent (i)  $d(x, x)$ = $0$, for every $x \in X$;

\noindent (i)  if $n, m \in X$ and $n  \not= m)$, then $d(n, m)$ = $2^{-m}$.

Consider the mapping $g: X \longrightarrow  X$, where $g(n )$ = $n+1$,  for every $n \in X$. 
Obviously, $Fix\,(g)$ = $\{x \in X: g(x ) = x\}$ = $\emptyset $. 

 Let  $O(x,g)$ = $\{x_n: n \in \mathbb N\}$ be the Picard orbit of $g$ 
 at the point $x$, i.e.,  $x_0$ = $x$
 and $x_n $ = $g^n(x)$,   for every $n \in \mathbb N$.
 
{\bf Property 1.} {\it If $n  \in X$, then $O(n, g)$ = $\{m \in X: m \geq  n\}$ is a Cauchy sequence 
and $\lim_{k\rightarrow \infty }g^k(n)$ =$m$ for each $m \in X$.}

By construction,  $\lim_{k\rightarrow \infty }d(m, g^k(n))$ = $\lim_{k \rightarrow \infty }2^{-k-n}$ = $0$.

{\bf Property 2.}  {\it $(X, d)$ is a  quasimetric space}.

If $n, m, k \in X$, then $d(n, m)$ + $d(m, k)$ = $2^{-m}$ + $2^{-k} > 2^{-k}$ = $d(n, k)$.
Hence $d$ is a quasimetric.

{\bf Property 3.}  {\it $(X, d)$ is a complete quasimetric space}.

{\bf Proof.} Let $\{x_n: n \in \omega \}$ be a  sequence.

{\bf Case 1.} {\it There exists $m \in \omega $ such that $x_n$ = $x_m$ for each $n \geq m$.}

In this case  $\lim_{n\rightarrow \infty }x_n$ = $x_m$ and   $\{x_n: n \in \omega \}$ is a
Cauchy convergent  sequence. 

{\bf Case 2.} {\it There exist two distinct numbers $m, k \in \omega $ such that  for each $n \in \omega $ there exist $m(n), k(n) \geq  n$
for which $x_m \not= x_k$, $x_{m(n)}$ = $x_m$ and  $x_{k(n)}$ = $x_k$. }

In this case     $\{x_n: n \in \omega \}$ is not a Cauchy  sequence and  is not a 
convergent  sequence. 

{\bf Case 3.} {\it  There exists a  number $m \in \omega $ such that:

- for each $n \in \omega $ there exists $m(n) \geq  n$
for which $x_{m(n)}$ = $x_m$;

- if $k \in \omega $ and $k \not= m$, then the set   $ \{n \in \omega : x_n $ = $x_k\}$ is finite.}

In this case  $\lim_{n\rightarrow \infty }x_n$ = $x_m$ and   $\{x_n: n \in \omega \}$ is a
Cauchy convergent  sequence.

{\bf Case 4.} {\it For each $m \in \omega $  the set   $ \{n \in \omega : x_n $ = $x_m\}$ is finite}.

In this case  $\lim_{n\rightarrow \infty }x_n$ = $x_m$ for each $m \in \omega $ and   $\{x_n: n \in \omega \}$ is a
Cauchy convergent  sequence.

{\bf Property 4.} {\it  $d(g(x),g(y))$ = $2^{-1}\cdot d(x,y)$, for all $x, y \in X$.}

{\bf Property 5.} {\it  $(X, \mathcal T(d))$ is a compact $T_1$-space and $\mathcal T(d)$ = $\{\emptyset \} \cup \{X \setminus F: F$ {\it is a finite set}$\}$.}
\end{exm}

 \begin{exm}\label{E3.3}  Let $X$ = $\mathbb N \cup \{\mu , \nu \}$ and $\mu , \nu   \not\in \mathbb N$. In $\mathbb N$ consider a sequence
 $\{i_n: n \in \mathbb N\}$ and  a sequence $\{k_n: n \in \mathbb N\}$ such that 

a) $1$ = $i_1$ and $i_n < k_n < i_{n+1}$, for each $n \in \mathbb N$;

b)  $\Sigma \{m^{-1}: m \in \mathbb N, i_n \leq  m  < k_n - 1 \} < 1$,    $\Sigma \{m^{-1}: m \in \mathbb N,  k_n + 1 < m \leq  i_{n+1}\} < 1$,
   $\Sigma \{m^{-1}: m \in \mathbb N,  i_n \leq m  <  k_n \} \geq  1$,    $\Sigma \{m^{-1}: m \in \mathbb N,  k_n <  m \leq i_{n+1}\} \geq  1$
   for each $n \in \mathbb N$. 

Consider on $\mathbb N$ the function $f(n)$ = $\Sigma \{m^{-1}: m \in \mathbb N, m \leq n\}$.
The set $I_n$ = $\{m \in \mathbb N: i_n \leq  m \leq  i_{n+1}\}$ is called an interval of integers. If $m \in I_n$, then:

i) $m$ is in the first part of the interval $I_n$ if $m < k_n$;

ii)  $m$ is in the second part of the interval $I_n$ if $m > k_n$;
 
iii)    $k_n$ is in the middle part of the interval $I_n$. 

Now we construct on $X$ the distance $d$ with the conditions:

{\it (C1)} $d(x, x)$ = $0$, for each $x \in X$;

{\it (C2)} $d(\mu , \nu )$ = $d(\nu , \mu )$ = $d(n, \mu  )$ = $d(n, \nu )$ = $1$  and $d(n, m)$ = $min \{1, |f(n) - f(m)|\}$, for all $n, m \in \mathbb N$;

{\it (C3)}  $d(\mu ,m)$ =  $ min \{1, \Sigma \{i^{-1}: i \in I_n, i_n \leq i \leq m\}\}$ 
and $d(\nu  ,m)$ =  $min \{1, \Sigma \{i^{-1}: i \in I_n, m <i \leq k_n\}$ 
if $m$ is in the first part of $I_n$;

{\it (C4)}   $d(\mu ,m)$ =  $ \{1, \Sigma \{i^{-1}: i \in I_n, m < i \leq i_{n+1}\}\}$  
and $d(\nu  ,m)$ =   $min \{1, \Sigma \{i^{-1}: i \in I_n, k_n\leq i \leq  m\}$ 
if $m$ is in the second part of $I_n$

{\it (C5)}   $d(\mu , k_n)$ = $1$ and $d(\nu , k_n)$ = $k_n^{-1}$. 

By construction, $0 \leq  d(x, y) \leq  1$, for all $x, y \in X$.

We put $\varphi (\mu )$ = $\mu $, $\varphi (\nu )$ = $\nu $ and $\varphi (n)$ = $n+1$   for each $n \in \mathbb N$. 
By construction, $Fix(\varphi )$ = $\{\mu , \nu \}$.

{\bf Property 1.} {\it  $(X, d)$ is a complete distance space}.

The space $(X, d)$ has not non-trivial Cauchy sequences, i.e., if $\{x_n \in X: n \in \mathbb N\}$ is a Cauchy sequence, then
there exists $m \in \mathbb N$ such that $x_m = x_n$, for all $n \geq m$ and $\lim_{n\rightarrow \infty }x_n$ = $x_m$.

{\bf Property 2.} {\it  $(X, d)$ is a  quasimetric space}.

Fix three distinct points $x, y, z   \in X$. We discuss the following cases.

{\bf Case 1.}  $x, y, z   \in \mathbb N$.

On $\mathbb N$ the distance $d$ is a metric. Hence $d(x,z) \leq d(x, y)$ + $d(y, z)$.

{\bf Case 2.}   $\{x, y\}$ =  $\{\mu , \nu \}$ and $z \in \mathbb N$.

In this case $d(x, z) \leq  1$ = $d(x,y) < d(x,y)$ + $d(y,z)$.

{\bf Case 3.}   $\{x, z\}$ =  $\{\mu , \nu \}$ and $y \in \mathbb N$.

In this case $d(x, z) \leq  1$ = $d(y,z) < d(x,y)$ + $d(y,z)$.

{\bf Case 4.}   $\{y, z\}$ =  $\{\mu , \nu \}$ and $x \in \mathbb N$.

In this case $d(x, z)$ = $1$ = $d(x,y) < d(x,y)$ + $d(y,z)$.

{\bf Case 5.}  $z \in \{\mu , \nu \}$, $x, y \in \mathbb N$.

In this case $d(x, z)$ = $d(y,z)$ = $1$ and $d(x,z) < d(x, y)$  + $d(y, z)$. 

{\bf Case 6.}   $y \in \{\mu , \nu \}$ and $x, z \in \mathbb N$.

In this case $d(x, z) \leq 1$, $d(x,y)$ = $1$ and $d(x,z) < d(x, y)$  + $d(y, z)$.

{\bf Case 7.}    $x \in \{\mu , \nu \}$, $n \in \mathbb N$ and $i_n \leq  z < y \leq  k_n$.

If  $x = \mu $, then $d(x,z) \leq  d(x, y)$  and  $d(x,y)$ + $d(y, z) \geq  d(x,z)$.

If $x = \nu  $ and $y < k_n$, then $d(x, z)$ = $\min \{1,  \Sigma \{i^{-1}: i \in I_n, z < i \leq k_n\}\}$ and  $d(x,y)$ + $d(y, z)$ = 
$  \min\{1, \Sigma \{i^{-1}: i \in I_n, y < i \leq k_n\}\}$ +  $  \min\{1, \Sigma \{i^{-1}: i \in I_n, z< i \leq y\}\} \geq  d(x,z)$.

If $x = \nu  $ and $y $ = $ k_n$, then $d(x, z)$ = $\min \{1,  \Sigma \{i^{-1}: i \in I_n, z < i \leq k_n\}\}$ and  $d(x,y)$ + $d(y, z)$ = 
$ k_n^{-1}$ +  $  \min\{1, \Sigma \{i^{-1}: i \in I_n, z< i \leq y\}\}$ = $ k_n^{-1}$ + $d(x,z) > d(x,z)$.

{\bf Case 8.}    $x \in \{\mu , \nu \}$, $n \in \mathbb N$ and $i_n \leq  y < z \leq k_n$.

If $x = \mu $, then $d(x, z)$ = $\min \{1,  \Sigma \{i^{-1}: i \in I_n, i_n \leq i \leq z\}\}$ and  $d(x,y)$ + $d(y, z)$ =
$  \min\{1, \Sigma \{i^{-1}: i \in I_n, i_n \leq i \leq y\}\}$ +  $ \min\{1,  \Sigma \{i^{-1}: i \in I_n, y< i \leq z\}\} \geq   d(x,z)$.

If  $x = \nu $, then $d(x,z) \leq  d(x, y)$  and  $d(x,y)$ + $d(y, z) \geq  d(x,z)$.

{\bf Case 9.}    $x \in \{\mu , \nu \}$, $n \in \mathbb N$ and $ k_n \leq  y < z \leq   i_{n+1}$.

If  $x = \mu  $, then $d(x,z) \leq  d(x, y)$  and  $d(x,y)$ + $d(y, z) \geq  d(x,z)$.

If $x = \nu  $, then $d(x, z)$ = $ \min\{1, \Sigma \{i^{-1}: i \in I_n, k_n \leq i \geq  z\}\}\}$ and  $d(x,y)$ + $d(y, z)$ = 
$  \min\{1, \Sigma \{i^{-1}: i \in I_n, k_n \leq i \geq  y\}\}$ +  $  \min\{1, \Sigma \{i^{-1}: i \in I_n, y< i \leq z\}\}$ =  $d(x,z)$.

{\bf Case 10.}    $x \in \{\mu , \nu \}$, $n \in \mathbb N$ and $ k_n \leq  z < y \leq  i_{n+1}$.

If $x = \mu $ and $ y<  i_{n+1}$, then $d(x, z)$ = $  \min\{1,  \Sigma \{i^{-1}: i \in I_n, z < i \leq k_{n+1}\}\}$ and  $d(x,y)$ + $d(y, z)$ = 
$  \min\{1, \Sigma \{i^{-1}: i \in I_n,  y <  i \leq k_{n+1}\}\}$ +  $  \min\{1, \Sigma \{i^{-1}: i \in I_n, z< i \leq y\}\} \geq d(x,z)$.

If  $x = \mu  $  and $ z$ =  $i_{n+1}$,  then $d(x, z)$ = $  \min\{1,  \Sigma \{i^{-1}: i \in I_n, z < i \leq k_{n+1}\}\}$ and  $d(x,y)$ + $d(y, z)$ = 
$  \min\{1, \Sigma \{i^{-1}: i \in I_n,  y <  i \leq k_{n+1}\}\}$ +  $ \min\{1,  \Sigma \{i^{-1}: i \in I_n, z< i \leq y\}\} \geq d(x,z)$.

If  $x = \nu $, then $d(x,z) \leq  d(x, y)$  and  $d(x,y)$ + $d(y, z) \geq  d(x,z)$.

{\bf Case 11.}    $x \in \{\mu , \nu \}$, $n \in \mathbb N$ and $i_n \leq y <  k_n  <  z \geq i_{n+1}$.

If $x = \mu $, then $d(x, y)$ = $\min \{1,  \Sigma \{i^{-1}: i \in I_n, i_n \leq i \leq y\}\}$,
$d(y, z)$ =   $ \min \{1,  \Sigma \{i^{-1}: i \in I_n, y< i \leq z\}\}$ and  $d(x,y)$ + $d(y, z) \geq 1$. Hence 
 $d(x,y)$ + $d(y, z)  \geq  d(x,z)$.

If  $x = \nu $, then $d(x,z) \leq  d(y, z)$  and  $d(x,y)$ + $d(y, z) \geq  d(x,z)$.

{\bf Case 12.}    $x \in \{\mu , \nu \}$, $n \in \mathbb N$ and $i_n \leq z <  k_n  y \geq i_{n+1}$.

If $x = \mu $, then $d(x, y)$ = $\min \{1,  \Sigma \{i^{-1}: i \in I_n, y \leq i \leq i_{n+1}\}\}$,
$d(y, z)$ =   $ \min \{1,  \Sigma \{i^{-1}: i \in I_n, z< i \leq y \}\}$ and  $d(x,y)$ + $d(y, z) \geq 1$. Hence 
 $d(x,y)$ + $d(y, z)  \geq  d(x,z)$.

If  $x = \nu $, then $d(x,z) \leq  d(y, z)$  and  $d(x,y)$ + $d(y, z) \geq  d(x,z)$.

{\bf Case 13.}    $x \in \{\mu , \nu \}$, $n \in \mathbb N$ and $k_n < y <  i_{n+1}  <  z \geq k_{n+1}$.

If $x = \mu $, then $d(x, z)$ = $\min \{1,  \Sigma \{i^{-1}: i \in I_{n+1},  i_{n+1} \leq  i \leq z\}\}$
$\leq \min \{1,  \Sigma \{i^{-1}: i \in \mathbb N,  y \leq  i \leq z\}\}\}$ = $d(y,z)$. Hence 
 $d(x,y)$ + $d(y, z)  \geq  d(x,z)$.

If  $x = \nu $, then $d(x, y)$ = $\min \{1,  \Sigma \{i^{-1}: i \in I_{n+1},  z <  i \leq k_{n+1}\}\}$,
$d(y, z)$ = $ \min \{1,  \Sigma \{i^{-1}: i \in \mathbb N,  y <  i \leq z\}\}\}$. Hence $d(x,y)$ + $d(y,z) \geq  1$
and $d(x,y)$ + $d(y, z)  \geq  d(x,z)$.

{\bf Case 14.}    $x \in \{\mu , \nu \}$, $n \in \mathbb N$ and $k_n < z <  i_{n+1}  <  y \geq k_{n+1}$.

If $x = \mu $, then $d(x, z)$ = $\min \{1,  \Sigma \{i^{-1}: i \in I_{n},  z < i \leq i_{n+1}\}\}$
$\leq \min \{1,  \Sigma \{i^{-1}: i \in \mathbb N,  z \leq  i \leq y\}\}\}$ = $d(y,z)$. Hence 
 $d(x,y)$ + $d(y, z)  \geq  d(x,z)$.

If  $x = \nu $, then $d(x, y)$ = $\min \{1,  \Sigma \{i^{-1}: i \in I_{n+1},  y <  i \leq k_{n+1}\}\}$,
$d(y, z)$ = $ \min \{1,  \Sigma \{i^{-1}: i \in \mathbb N,  z <  i \leq y\}\}\}$. Hence $d(x,y)$ + $d(y,z) \geq  1$
and $d(x,y)$ + $d(y, z)  \geq  d(x,z)$.

{\bf Case 15.}    $x \in \{\mu , \nu \}$, $n \in \mathbb N$ and $i_n \leq  y  \leq  k_n <  i_{n+1}  <  z \geq k_{n+1}$ or
 $i_n \leq  z  \leq  k_n <  i_{n+1}  <  y \geq k_{n+1}$.

In this case $d(y,z)$ = $1$ and  $d(x,y)$ + $d(y, z)  \geq  d(x,z)$.

There are no other possible cases. The proof of Property 2 is complete.

{\bf Property 3.} {\it The mapping $\varphi $ is continuous, $d(\varphi (x), \varphi (y)) \leq  2\cdot d(x,y)$ for all $x, y \in X$
 and $d(\varphi (x), \varphi (y)) <  d(x,y)$ for all distinct points $x, y \in \mathbb N$.}
 
 If $x  \in \{\mu , \nu \}$ and $n \in \mathbb N$, then  $|d(\varphi (x), \varphi (n)) - d(x, n)|$ = $|d(x, n+1) - d(x,n)| \leq n^{-1}$.

{\bf Property 4.} {\it If $x \in X$, then $\lim_{n\rightarrow \infty }d(\varphi ^n(x), \varphi ^{n+1}(x))$ =  $0$.}

{\bf Property 5.} {\it The space $(X, \mathcal T(d))$ is complete metrizable.}

If $x \in \mathbb N$, then $N_nx$ = $\{x\}$ for each $n \in \mathbb N$. If $x  \in \{\mu , \nu \}$ and $n \in \mathbb N$, then 
$O_nx$ = $\{y \in X: d(x,y) < 2^{-n-2}\}$. Then $\mathcal B$ = $\{O_nx: x \in X, n\in  \mathbb N\}$ is a base of open-and-closed subsets
of the space $(X, \mathcal T(d))$. The proof is complete. 

{\bf Property 6.} {\it There exists a closed discrete sequence $\{x_n \in \mathbb N: n \in \mathbb N\}$ of the space $(X, \mathcal T(d))$ 
such that $x_n < x_{n+1}$ for each $n \in \mathbb N$.}

For each $n \in \mathbb N$ fix $x_n \in I_{n+2}$ such that $\Sigma \{i^{-1}: 1/4 \leq  i_{n+2} \leq i \leq x < 3/4\}$.

{\bf Property 7.} {\it For each $n \in \mathbb N$ the points $\mu $, $\nu $ are points of accumulation of the 
Picard orbit $O(x, \varphi )$.}

{\bf Property 8.} {\it The orbit $O(1, \varphi )$ = $n \in \mathbb N$ is not convergent in $(X, d)$. }

 \begin{exm}\label{E3.4}  Let $\omega $ = $\{0, 1, 2, ...\}$ and $\omega $ be the first infinite ordinal number, $\Omega $ be 
 be the first uncountable ordinal number. For any ordinal number $\alpha $ there exist a unique limit  ordinal number
$l(\alpha )$ and a unique integer $i(\alpha ) \in \omega $ such that $l(\alpha ) \leq  \alpha $ and $\alpha $ = $l(\alpha ) + i(\alpha )$.
If $l(\alpha )$ = $\alpha $, then $\alpha $ is a limit  ordinal.
Let $l^\prime (\alpha )$ = $min \{\beta  \in X: \alpha  < \beta, \beta = l(\beta )\}$. 

Denote by $X$ = $\{\alpha : \alpha  < \Omega \}$ the set of all  countable ordinal numbers.

Consider the mapping $g: X \longrightarrow  X$, where $g(\alpha )$ = $\alpha +1$,  for every $\alpha  \in X$.

By construction, $Fix(g)$ = $\{\alpha  \in X: g(\alpha ) = \alpha \}$ = $\emptyset $ and $l(g(\alpha ))$ = $l(\alpha )$,
 $i(g(\alpha ))$ = $i(\alpha ) + 1$   for every $\alpha  \in X$.
 Let $g^1$ = $g$ and $g^{n+1}$ = $g\circ g^n$ for each $n \in \mathbb N$ = $\{1, 2, ...\}$. If $x \in X$, then $x_0$ = $x$
 and $x_n $ = $g^n(x)$   for every $n \in \mathbb N$. The set $O(x,g)$ = $\{x_n: n \in \mathbb N\}$ is the Picard orbit
 of the point $x$.
 If   $\alpha , \beta  \in X$, $\alpha  < \beta $  and $l(\beta )$ = $l(\alpha )$, then $\beta \in O(\alpha ,g)$.

 On $X$ consider the
distance $d$ with the conditions:

- $d(\alpha , \alpha )$ = $0$ for every $\alpha  \in X$;

- if $\alpha , \beta  \in X$  and $l(\beta )$ = $l(\alpha )$, then $d(\alpha , \beta )$ = $|2^{-i(\alpha )} - 2^{-i(\beta )}|$;

-  if $\alpha , \beta  \in X$  and $l(\beta ) < l(\alpha )$, then $d(\alpha , \beta )$ = $ 2^{-i(\beta )}$
   and  $d(\beta ,\alpha )$ = $1 + 2^{-i(\alpha )}$.
   
{\bf Property 1.} {\it If $\alpha  \in X$, then:

- $d $ is a metric on the orbit $O(\alpha ,g)$ and $d(g(x),g(y))$ = $2^{-1}d(x,y)$ for all $x, y \in O(\alpha , g)$;

- the orbit $O(\alpha , g)$ =  $\{\alpha _n = g^n(\alpha ): n \in \mathbb N\}$ is a  fundamental sequence in $(X, d)$;

- if $\beta  > \alpha $ and $l(\beta ) \geq l^\prime (\alpha ) > \alpha \geq l(\alpha )$, then $\beta $ is a limit point
of the sequence  $\{\alpha _n : n \in \mathbb N\}$;

- if $l(\beta )$ =  $l(\alpha )$ , then $\beta $ is not a limit point
of the sequence  $\{\alpha _n : n \in \mathbb N\}$.}

{\bf Property 2.} {\it Assume that $\{\alpha _n \in X: n \in \mathbb N\}$ is a convergent sequence in $(X, d)$ and
$\alpha $ = $min \{\beta : \beta  = lim_{n \rightarrow \infty }\alpha _n\}$, $\check{\alpha } $ = $\sup\{l(\alpha _n) : n \in \mathbb N\}$,
$\vec{\alpha }$  = $\sup\{l^\prime (\alpha _n) : n \in \mathbb N\}$. 

1. In $X(\omega )$ = $\omega \cup \{\omega \}$ there exists  the limit $b$ = $\lim_{n\rightarrow \infty } i(\alpha _n)$. 

2. If  $\check{\alpha } < \vec{\alpha }$, then  $\{\alpha _n : n \in \mathbb N\} \setminus O(\check{\alpha }, g) $ is a finite
set  $\alpha  \in  O(\check{\alpha }, g) $ and $b < \omega $.

3.  If  $\check{\alpha }$ = $ \vec{\alpha }$, then   $\alpha $ = $  \vec{\alpha }$ and $b$ = $\omega $.}

{\bf Property 3.} {\it $(X, d)$ is a complete quasimetric space}.

{\bf Proof.} Completeness follows from the above properties.

Fix $\alpha , \beta , \gamma    \in X$.

{\bf Case 1.} $l(\alpha )$ = $l(\beta )$ = $l(\gamma )$.

In this case  $\alpha , \beta , \gamma    \in O(l(\alpha ), g)$ and $d(\gamma , \alpha )$ 
= $d(\alpha , \gamma ) \leq d(\alpha , \beta ) $ + $d(\beta , \gamma )$.

{\bf Case 2.} $l(\alpha )= l(\beta ) < l(\gamma )$.

In this case   $d(\alpha , \gamma )$ = $1+ 2^{-i(\gamma )} <  d(\alpha , \beta ) + 1 +  2^{-i(\gamma )} $ = $d(\alpha , \beta ) $ + $d(\beta , \gamma )$.

{\bf Case 3.} $l(\gamma ) < l(\alpha ) = l(\beta )$.

In this case  $d(\alpha , \gamma )$ = $d(\beta , \gamma )$ = $2^{-i(\gamma )}$ and  $d(\alpha , \gamma ) \leq  d(\alpha , \beta ) $ + $d(\beta , \gamma )$.

{\bf Case 4.} $l(\alpha ) = l(\gamma  ) < l(\beta )$.

In this case   $d(\alpha , \gamma ) \leq  1 <  1 +  2^{-i(\beta  )}$ = $d(\alpha , \beta ) \leq  d(\alpha , \beta ) $ + $d(\beta , \gamma )$.

{\bf Case 5.} $l(\beta ) < l(\alpha ) = l(\gamma )$.

In this case   $d(\alpha , \gamma )$ = $|2^{-i(\alpha )} - 2^{-i(\gamma )}|  < 1 < d(\beta  , \gamma ) < d(\alpha , \beta ) $ + $d(\beta , \gamma )$.

{\bf Case 6.} $l(\alpha ) < l(\beta )=  l(\gamma )$.

In this case   $d(\alpha , \gamma )$ = $1  +  2^{-i(\gamma )} \leq   1  +  2^{-i(\beta )}+ |2^{-i(\beta )} - 2^{-i(\gamma )}| $ 
= $d(\alpha , \beta ) $ + $d(\beta , \gamma )$.

{\bf Case 7.}  $l(\beta )$ = $l(\gamma ) < l(\alpha )$.

In this case  $d(\alpha , \beta  ) $ = $2^{-i(\beta )}$, $d(\alpha , \gamma )$ = $2^{-l(\gamma )}$ and  
 $d(\alpha , \beta ) $ + $d(\beta , \gamma )$ = $2^{-l(\beta )}$ + $|2^{-l(\beta )} - 2^{-l(\gamma )}| \geq  2^{-l(\gamma )}$
=  $d(\alpha , \gamma )$.

{\bf Case 8.} $l(\alpha ) < l(\beta ) < l(\gamma )$.

In this case   $d(\alpha , \gamma )$ = $1+ 2^{-i(\gamma )} <  d(\alpha , \beta ) + 1 +  2^{-i(\gamma )} $ = $d(\alpha , \beta ) $ + $d(\beta , \gamma )$.

{\bf Case 9.} $l(\alpha ) <  l(\gamma  ) < l(\beta )$.

In this case   $d(\alpha , \gamma )$ =  $1 +  2^{-i(\gamma )} <  1  +  2^{-i(\beta )} +   2^{-i(\gamma )}$ = $d(\alpha , \beta ) $ + $d(\beta , \gamma )$.

{\bf Case 10.} $l(\beta ) < l(\alpha ) < l(\gamma )$.

In this case   $d(\alpha , \gamma )$ = $1 + 2^{-i(\gamma )} <  d(\alpha , \beta ) + 1 +  2^{-i(\gamma )}$ =  $d(\alpha , \beta ) $ + $d(\beta , \gamma )$.

{\bf Case 11.}  $l(\beta ) < l(\gamma ) < l(\alpha )$.

In this case    $d(\alpha , \gamma ) \leq 1 <  d(\alpha , \beta ) + 1 + 2^{-i(\gamma )}$ = $ d(\alpha , \beta ) $ + $d(\beta , \gamma )$.

{\bf Case 12.} $l(\gamma ) < l(\alpha ) <  l(\beta )$.

In this case  $d(\alpha , \gamma )$ = $d(\beta , \gamma )$ = $2^{-i(\gamma )}$ and  $d(\alpha , \gamma ) \leq  d(\alpha , \beta ) $ + $d(\beta , \gamma )$.

{\bf Case 13.} $l(\gamma ) <  l(\alpha ) < l(\beta )$.

In this case  $d(\alpha , \gamma )$ =   $d(\beta  , \gamma )  <  d(\alpha , \beta ) $ + $d(\beta , \gamma )$.

The proof is complete.

{\bf Property 4.} {\it  $d(g(x),g(y)) < d(x,y)$, for all $x, y \in X$, $x\neq y$.}

{\bf Property 5.} {\it  If $n \in \omega $, then  $X_n$ = $\{\alpha  \in X: i(\alpha ) \leq  n\}$ is  a closed discrete metrizable
subspace of the  space $X$. Moreover, $d(x, y) \geq 2^{-n}$ for all distinct points $x, y \in X_n$ and the set $X \setminus X_n$ is open
and dense in $X$.}
\end{exm}

\end{exm}

 \section{Fixed points and dislocated completeness of distance spaces}

Let $(X, d)$ be a distance space. We denote by $d_s(x,y)$ = $d(x,y) + d(y,x)$,  the symmetric associated to the distance $d$.
The spaces $(X, d)$ and $(X, d_s)$ share the same Cauchy sequences. If $d$ is a quasimetric, then $d_s$ is a metric.

Some authors, instead of the  conditions of uniqueness  of the limit of the Cauchy sequence introduced the concept of a 
stronger limit, i.e., the concept of a dislocated  convergence  of the  sequence (see \cite{AS, D, SA, ZHA}). It is easy to see that dislocated convergence is implicitly
a variant of the symmetry of the distance.

A sequence $\{x_n \in X: n \in \mathbb N\}$ is  said to be dislocated  convergent
to $x  \in X$ if $\lim n_{n\rightarrow \infty }(d(x_n, x) + d(x, x_n))$ = $0$ and we denote this $s$-$\lim_{n\rightarrow \infty }x_n$ = $x$.

The distance space $(X, d)$ is   dislocated  complete if any Cauchy sequence of $X$ is    dislocated  convergent in $(X, d)$.

The distance spaces from Examples \ref{E3.1} and \ref{E3.2} are complete non-dislocated  complete.

The space $(X, d)$ is dislocated  complete if  and only if the space $(X,d_s)$ is complete. A symmetric space  
is dislocated  complete if  and only if it is complete.

 \begin{lemma}\label{L4.1}  \label{lem}   Let $d$ be an $N$-distance  on a space $X$. If  $\{x_n \in X: n \in \mathbb N\}$ is  dislocated  convergent
 sequence, then it is  dislocated  convergent to a unique point.
\end{lemma}

\begin{proof} Assume that  $s-\lim_{n\rightarrow \infty }x_n$ = $x$ and  $s-\lim_{n\rightarrow \infty }x_n$ = $y$. 
Suppose that $d(x, y)$ = $4 \varepsilon  > 0$.  There exists a number $\delta $ such that:

- if $d(x,u) \leq  \delta $  and  $d(u,v) \leq \delta $, then $d(x,v) \leq  \varepsilon $;

- if $d(y,u) \leq  \delta $  and  $d(u,v) \leq \delta $, then $d(y,v) \leq  \varepsilon $.

Since  $\lim_{n\rightarrow \infty }(d(x_n, x) + d(x, x_n))$ = $0$ 
and  $\lim_{n\rightarrow \infty }(d(x_n, y) + d(y, x_n))$ = $0$, there exists $m \in \mathbb N$ such that
$d(x_n, x) + d(x, x_n)) < \delta $ and  $d(x_n, y) + d(y, x_n)) < \delta $, for each $n \geq  m$.
Hence $d(x, x_m) \leq  \delta $, $d(x_m, y) < \delta $ and $d(x,y) > \varepsilon $, a contradiction.
Therefore $d(x,y)$ = $d(y,x)$ = 0 and $x = y$. %The proof is complete.
 \end{proof}

In view of Lemma \ref{lem}, most of the problems on fixed points in  dislocated  complete quasimetric spaces could be reduced to the case
of complete metric spaces. 

For example, if $g: X \longrightarrow  X$ is a contraction on a dislocated  complete quasimetric space $(X, d)$, i.e., there exists  $0\leq  \lambda  < 1$, such that 
$$
d(g(x), g(y)) \leq  \lambda \cdot d(x, y), \textnormal{ for all } x, y \in X,
$$
 then $(X, d_s)$ is a complete metric space and 
  $$
  d_s(g(x), g(y)) \leq  \lambda \cdot d_s(x, y), \textnormal { for all } x, y \in X.
  $$
 Hence, see also our results in Section \ref{sect}, by classical contraction principle, $g$ has a unique fixed point and every Picard orbit is a Cauchy sequence which is dislocated convergent to the fixed point of $g$.

\section*{Acknowledgements}

The  second author acknowledges the support provided by the Deanship of Scientific Research at King Fahd University of Petroleum and Minerals for funding this work through the projects IN151014 and IN141047.

\vskip 0.25 cm {\it $^{a}$ Department of Physics, Mathematics and Information
Technologies\newline
\indent Tiraspol State University\newline 
\indent  Gh. Iablocikin 5., MD2069 Chi\c sin\u au, Republic of Moldova

E-mail: mmchoban@gmail.com }

\vskip 0.5 cm {\it $^{b}$ Department of Mathematics and Computer Science

North University Center at Baia Mare 

Technical University of Cluj-Napoca

Victoriei 76, 430072 Baia Mare ROMANIA

E-mail: vberinde@cunbm.utcluj.ro}

\vskip 0.25 cm {\it $^{c}$ Department of Mathematics and Statistics

King Fahd University of Petroleum and Minerals

Dhahran, Saudi Arabia

E-mail: vasile.berinde@gmail.com}
\end{document}